\documentclass{amsart}

\usepackage{graphicx}
\usepackage{amssymb, amsmath, amsthm}
\usepackage{color}
\usepackage{hyperref}

\newtheorem{lemma}{Lemma}
\newtheorem{proposition}[lemma]{Proposition}
\newtheorem{theorem}[lemma]{Theorem}

\newcommand{\R}{\mathbb R}

\newcommand{\ve}{\varepsilon}

\title[MCF evolving away from heat flow]{Mean curvature flow of entire graphs evolving away from the heat flow}

\author{Gregory Drugan}
\address{Department of Mathematics, University of Oregon, Eugene, OR 97403}
\email{drugan@uoregon.edu}

\author{Xuan Hien Nguyen}
\address{Department of Mathematics, Iowa State University, Ames, IA 50011}
\email{xhnguyen@iastate.edu}


\commby{Lei Ni}

\begin{document}

\maketitle

\begin{abstract}
We present two initial graphs over the entire $\R^n$, $n \geq 2$ for which the mean curvature flow behaves differently from the heat flow. In the first example, the two flows  stabilize at different heights. With our second example, the mean curvature flow oscillates indefinitely while the heat flow stabilizes. 
These results highlight the difference between dimensions $n \geq 2$ and dimension $n=1$, where Nara--Taniguchi proved that entire graphs in $C^{2,\alpha}(\R)$ evolving under curve shortening flow converge to solutions to the heat equation with the same initial data. 
\end{abstract}

\section{Introduction}

We consider the solution $u:  \R^n \times [0,T) \to \R$ to the {\bf graphical mean curvature flow (MCF)},
	\begin{gather}
	\label{eq:mcf-nonparam}
	\displaystyle \frac{\partial u}{\partial t} = \sqrt{1 + |Du|^2}\, \textrm{div} \left(\frac{Du}{\sqrt{1+ |Du|^2}} \right), \\
	u(\cdot, 0 ) = u_0, \notag
	\end{gather}
where $u_0\in C^{2,\alpha}(\mathbb R^n), \alpha \in (0,1]$, and compare it to the solution $v : [0, T) \times \R^n \to \R$ to the heat equation with the same initial graph,
	\begin{gather}
	\label{eq:heat}
	\displaystyle \frac{\partial v}{\partial t} = \Delta v,\\
	v(\cdot, 0) = u_0. \notag
	\end{gather}
When $n=1$, the mean curvature flow is called curve shortening flow and equation \eqref{eq:mcf-nonparam} reduces to 
	$
	\frac{\partial u}{\partial t} = \frac{ u_{xx}}{1+ (u_x)^2}.
	$
For this dimension, Nara and Taniguchi~\cite{NT} showed that if the initial graph $u_0$ is in $C^{2,\alpha}(\R)$, the solution to the curve shortening flow converges to the solution to the heat equation with the same initial data. The proof relies heavily on the existence of a Gauss kernel and the fact that the right-hand side is a potential. 

In higher dimensions ($n\geq 2$), a couple of results suggest a similar behavior. If the initial graph $u_0$ is compactly supported, an argument by Huisken using large balls as barriers and the maximum principle shows that $u(\cdot, t)$ goes to zero uniformly as $t \to \infty$, which is also the asymptotic behavior of $v(\cdot, t)$  (see \cite{clutterbuck-schnurer-2011} for the mean curvature flow and \cite{repnikov-eidelman} for the heat equation). For bounded and radially symmetric entire initial graphs, Nara~\cite{nara;siam2008} proved the convergence of solutions to the mean curvature flow to solutions to the heat equation again.  

The object of this note is to show that the result of Nara--Taniguchi does not extend to higher dimensions in general.
	\begin{theorem}
	\label{thm:main}
	For $n \geq 2$, for every constant $\ve>0$, there is a function $u_0 \in C^{2,\alpha}(\R^n)$, $\alpha \in (0,1]$, for which the solution $u(\cdot, t)$  to the mean curvature flow \eqref{eq:mcf-nonparam} with initial data $u_0$ stabilizes to a constant $c_0 < \ve$, while the solution $v(\cdot, t)$  to the heat equation  \eqref{eq:heat} with initial data $u_0$ stabilizes to the constant 1. More precisely, we have
	\begin{equation}
	\label{eq:utoc0}
		u(\cdot, t) \to c_0 \text{ uniformly in $\mathbb R^n$ as } t\to \infty
	\end{equation}	
and  
	\begin{equation}
	\label{eq:vto1}
	v(\cdot, t) \to 1 \text{ uniformly in $\mathbb R^n$ as } t\to \infty.
	\end{equation}
\end{theorem}
	\begin{theorem}
	\label{thm:mainosc}
	For $n \geq 2$, for every $\ve >0$, there is a function $w_0 \in C^{2,\alpha}(\R^n)$, $\alpha \in (0,1]$, for which the solution $w(\cdot, t)$  to the mean curvature flow \eqref{eq:mcf-nonparam} with initial data $w_0$ oscillates
	\begin{equation}
	\label{eq:w:infsup}
	\liminf_{t\to \infty} w(0,t) \leq \ve, \quad \limsup_{t\to \infty} w(0,t) \geq 1,
	\end{equation}
and the solution $v(\cdot,t)$ to the heat flow \eqref{eq:heat} with initial data $w_0$ tends to $1$ uniformly on $\mathbb R^n$ as $t \to \infty$.
	\end{theorem}
These theorems are reminiscent of the different behaviors of the mean curvature flow of embedded initial data when $n \geq 2$ or when $n=1$. Indeed, Grayson proved that embedded curves remain embedded under mean curvature flow \cite{grayson;embedded-curves}, but embedded surfaces can pinch off \cite{grayson;evolution-surfaces}. 

The long time existence of smooth solutions $u$ for entire Lipschitz graphs $u_0$ was established by Ecker--Huisken~\cite{EH}. Looking at interior estimates for the mean curvature flow from the same authors~\cite{ecker-huisken;interior-estimates} and standard theory for the heat equation, one can see that both $u$ and $v$ move more slowly as time goes to infinity. Therefore, the flows cannot separate much as time gets large and what happens right at the start is critical. Both our functions $u_0$ and $w_0$ illustrate  how the mean curvature flow can pull away from the heat flow instantly. The idea for $u_0$ is to build tall thin spikes that enclose a large volume. The solution to the heat equation converges to the average of the volume enclosed between the graph of $u_0$ and the base $n$-plane, while the evolution by mean curvature loses a lot of volume at the start and will tend to a smaller constant. We use a similar idea for $w_0$. Starting with an oscillating graph, we add spikes in the lower regions, thereby adding volume. This forces the heat flow to stabilize but does not affect the mean curvature flow much.  We use shrinking doughnuts and spheres as barriers.

\vspace{7pt}

\noindent{\bf Acknowledgment.} The authors would like to thank Bruce Kleiner for suggesting the first example  during the workshop ``Geometric flows and Riemannian geometry" at the American Institute of Mathematics (AIM) in September 2015. We would like to thank the AIM for the hospitality and the organizers of the workshop, Lei Ni and Burkhard Wilking, for the invitation.


\section{Self-shrinking torus}
\label{sec:torus}
In~\cite{Ang}, Angenent constructed an embedded rotationally symmetric torus in $\mathbb{R}^{n+1}$ that shrinks self-similarly under mean curvature flow to a point in finite time. We denote both this torus and its parametrization by $\mathbb{T}$.  The standard scaling is such that $\mathbb{T}: S^{n-1} \times S^1 \to \mathbb{R}^{n+1}$ is a solution to the self-shrinker equation 
	\begin{equation}
	\label{eq:self-shrinker}
	{\bf H}_\mathbb{T} = -\frac{1}{2} \mathbb{T}^\perp,
	\end{equation}
where ${\bf H}$ is the mean curvature and $\perp$ is the projection onto the normal space. For what follows, we require only the existence of such a shrinking torus; the uniqueness of $\mathbb{T}$ is still an open problem.

For applications in the following sections, we will work with a rescaling of $\mathbb{T}$, which we denote by $\mathbb{T}_0$. The scale is chosen so that $\mathbb{T}_0$ is in the cylinder of radius $1/2$ around its axis of rotation and its height is less than $\ve$. Notice that any rescaling of $\mathbb{T}$ also shrinks self-similarly under mean curvature flow to the origin in finite time, although it no longer satisfies the self-shrinker equation \eqref{eq:self-shrinker}.

To determine the scaling for $\mathbb{T}_0$, we consider its profile curve in the half-plane $\{(r,z) \in \mathbb{R}^2 \, | \, r >0\}$, which we denote by $\Gamma_{\mathbb{T}_0}$. It follows from the construction in~\cite{Ang} that $\Gamma_{\mathbb{T}_0}$ is a simple closed convex curve; it intersects the $r$-axis perpendicularly at two points; and it is symmetric with respect to reflections about the $r$-axis. Let $(\ell_0,0)$ and $(r_0,0)$ be the two points where $\Gamma_{\mathbb{T}_0}$ intersects the $r$-axis, and choose $\ell_0 < r_0$. Also, let $\delta_0$ denote the maximum $z$-coordinate of $\Gamma_{\mathbb{T}_0}$. We say that $\ell_0$ is the {\bf inner radius}, $r_0$ is the {\bf outer radius}, and $\delta_0$ is the {\bf maximum height} of the torus $\mathbb{T}_0$ (see Figure \ref{fig:doughnut}). By our choice of scaling, we have $r_0 < 1/2$ and $\delta_0 <\ve$.

Picturing $\mathbb{T}_0$ in $\mathbb{R}^{n+1}$, with coordinates $\{x_1, \dots, x_n,z\}$, so that it has rotational symmetry about the $z$-axis and reflection symmetry about the base $n$-plane $\{z=0\}$, we can interpret $\ell_0$, $r_0$, and $\delta_0$ in the following way. The projection of $\mathbb{T}_0$ into the base $n$-plane is the annulus $\{ \ell_0\leq |x| \leq r_0 \}$, and $\mathbb{T}_0$ is sandwiched between the planes $\{z =\delta_0\}$ and $\{z = - \delta_0\}$.  

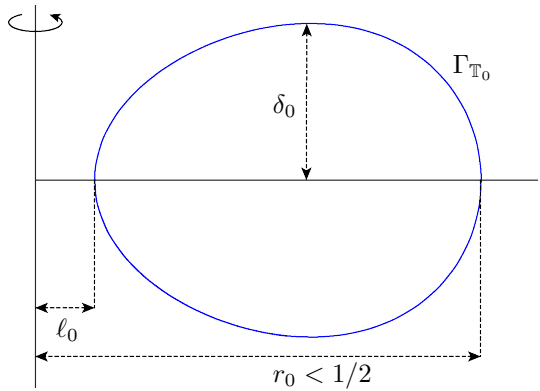
\begin{figure}[h]
\includegraphics[height=2in]{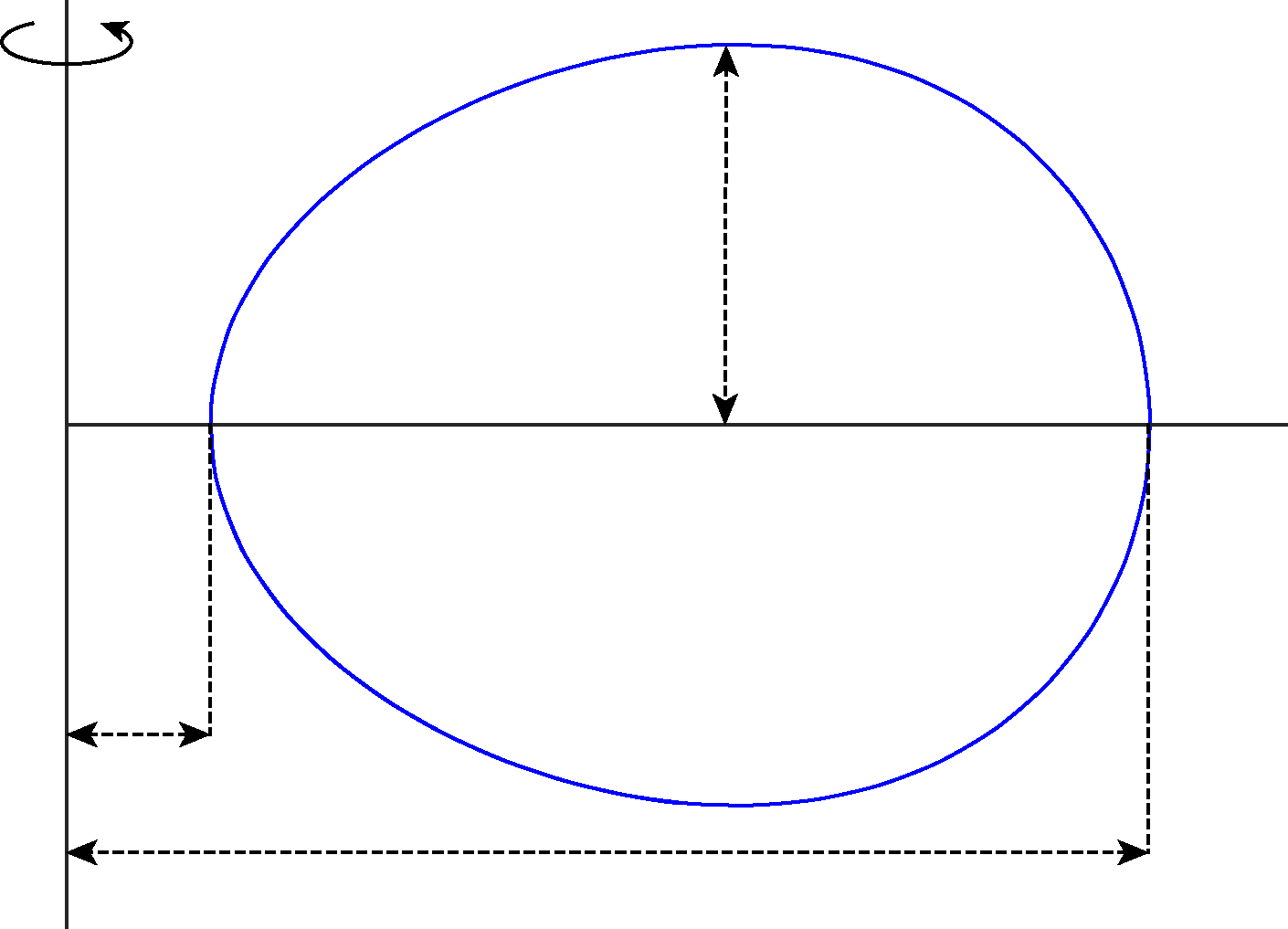}
\caption{A cross section of $\mathbb{T}_0$.}
\label{fig:doughnut}
\begin{picture}(0,0)
\put(-82,55){$\ell_0$}
\put(0,38){$r_0<1/2$}
\put(0, 140){$\delta_0$}
\put(68,155){$\Gamma_{\mathbb{T}_0}$}
\end{picture}
\end{figure}


\section{A periodic graph with thin spikes.} 
In this section, we construct the initial graph $u_0 \in C^{2,\alpha}( \mathbb R^n)$, which has tall thin spikes on the integer lattice, then study its evolution under mean curvature flow and heat flow.

\subsection{The initial graph $u_0$}
Consider the torus $\mathbb{T}_0$ described in the previous section and translate it so that it is centered at the point $(0,\delta_0) \in \mathbb{R}^n \times \mathbb{R}$. This translated torus shrinks self-similarly under mean curvature flow to its center in finite time. Using the property that $r_0 < 1/2$, we see that this torus is contained in the box $[-\frac{1}{2}, \frac{1}{2}]^n \times [0,2\delta_0]$. We also note that the projection of this torus onto the base $n$-plane is the annulus $\overline{B_{r_0}(0)} \, \backslash B_{\ell_0}(0) \subset \mathbb{R}^n$.

For each point $(m,\delta_0) \in \mathbb{Z}^n \times \mathbb{R}$, we place a copy of $\mathbb{T}_0$ centered at the point $(m, \delta_0)$ and denote it by $\mathbb{T}_{(m,\delta_0)}$. Because $r_0 < 1/2$, these tori are disjoint. We introduce two subsets of $\mathbb{R}^n$, which we identify with the base $n$-plane in $\mathbb{R}^{n+1}$. First, we let $\Omega_0$ be the region in the base $n$-plane obtained by removing the union of discs $B_{r_0}(m)$: $$\Omega_0 = \mathbb{R}^n \, \backslash \bigcup_{m \in \mathbb{Z}^n} B_{r_0}(m).$$ Second, we let $U_0$ be the union of the discs $B_{\ell_0}(m)$: $$U_0 = \bigcup_{m \in \mathbb{Z}^n} B_{\ell_0}(m).$$

Now, we are ready to define an entire function $u_0$ supported in the region $U_0$. We first define $u_0$ on the disc $B_{\ell_0}(0)$ to be a smooth non-negative function such that 
	\begin{gather*}
	u_0(x) = h_0, \quad |x| < \ell_0 / 3,\\
	u_0(x) = 0, \quad 2 \ell_0/3 < |x| <\ell_0,\\
	\int_{B_{\ell_0}} u_0 \, dvol=1,
	\end{gather*}
Next, we extend $u_0$ to be zero on the remainder of the square $[-\frac{1}{2}, \frac{1}{2}]^n$: $$u_0 = 0 \textrm{ on } [-\tfrac{1}{2}, \tfrac{1}{2}]^n \, \backslash B_{\ell_0}(0,0).$$ Finally, we define $u_0$ on all of $\mathbb{R}^n$ by making it periodic: $$u_0(x+m) = u_0(x), \textrm{ for } m \in \mathbb{Z}^n.$$ 

The graph of $u_0$ may be viewed as a collection of tall spikes. Each spike is concentrated around some point $m\in \mathbb{Z}^n$, and assuming $h_0> 2 \delta_0$, the spike passes through the hole of $\mathbb{T}_{(m,\delta_0)}$. In the next section, we describe the information these carefully placed tori give us on $u(\cdot, t)$.


\subsection{Evolution of the graph by mean curvature flow}

We begin this section by noting that for each $m \in \mathbb{Z}^n$, the torus $\mathbb{T}_{(m,\delta_0)}$ sits on or above the graph of $u_0$. Next, we run mean curvature flow. On one hand, because $u_0$ is a smooth entire function, the mean curvature flow with initial data $M_0 = \textrm{graph} \, u_0$ has a smooth solution $M_t = \textrm{graph} \, u(\cdot, t)$ that exists for all $t > 0$ (see~\cite{Eck},\cite{EH}), and an application of the strong maximum principle shows $u(\cdot,t) > 0$ for $t>0$. On the other hand, the torus  $\mathbb{T}_{(m,\delta_0)}$ shrinks to its center under the mean curvature flow in finite time $t_* < \infty$. It follows from the comparison principle for mean curvature flow that for $t>0$, the evolution of the torus  $\mathbb{T}_{(m,\delta_0)}$  is disjoint from the evolution of the entire graph $M_0$. In particular, by the time $t= t_*$, each tall spike of the original entire graph $M_0$ must pass through the hole of the surrounding torus. We will show that by the time $t=t_*$, the maximum height of $u(\cdot, t)$ is less than or equal to $\delta_0$.

The torus $\mathbb{T}_0$ evolves under mean curvature flow as a family of self-similar tori $\{ \mathbb{T}_t \}_{t \in [0,t_*)}$ that is shrinking to its center as $t$ approaches $t_*$. Let $\ell_t$, $r_t$, and $\delta_t$ denote the inner radius, outer radius, and maximum height of the torus $\mathbb{T}_t$, respectively, and set $$\Omega_t = \mathbb{R}^n \, \backslash \bigcup_{m \in \mathbb{Z}^n} B_{r_t}(m),$$ $$U_t = \bigcup_{m \in \mathbb{Z}^n} B_{\ell_t}(m).$$

\begin{proposition}
\label{prop:2}
For $t < t_*$, we have $u(\cdot, t) < \delta_0$ on $\Omega_t$.
\end{proposition}
\begin{proof}
This is true when $t=0$ by definition of $u_0$, and by comparison with the tori barriers, we have $u(\cdot, t) < \delta_0$ on $\partial \Omega_t$, for $t<t_*$. An application of the maximum principle shows that $u(\cdot, t) < \delta_0$ on $\Omega_t$ when $t < t_*$. Because the domain we are working in is unbounded, we provide the details for the reader's convenience. Using the continuity of $u$ and the periodicity of $u(\cdot,t)$, we see that the supremum of $u(\cdot,t)$ over the closed set $\Omega_t \cup \partial \Omega_t$ is achieved and depends continuously on $t$. If the proposition does not hold, then there will be a first time $t'<t_*$ where the supremum of $u(\cdot,t')$ over $\Omega_{t'} \cup \partial \Omega_{t'}$ equals $\delta_0$. Given the height estimate on the boundary $\partial \Omega_{t'}$, this supremum will be achieved at an interior point $x' \in \Omega_{t'}$. Consequently,  we have $u(x',t') = \delta_0$, $u(\cdot,t') \leq \delta_0$ in $\Omega_{t'}$, and $u(\cdot,t) < \delta_0$ in $\Omega_t$ for $t<t'$, which contradicts the maximum principle~\cite[Chapter 2]{Frd}.
\end{proof}

\begin{proposition}
\label{prop:3}
For $t \geq t_*$, we have $u(\cdot, t) \leq \delta_0$ on $\mathbb{R}^n$.
\end{proposition}
\begin{proof}
Using the previous proposition, we know that $u(\cdot, t) < \delta_0$ on $\Omega_t$ for $t < t_*$. In addition, as $t \to t_*$, we know that $\Omega_t$ converges to $\mathbb{R}^n$. Then, from the continuity of the solution, we deduce that $u(\cdot,t_*) \leq \delta_0$ on $\mathbb{R}^n$. By the maximum principle, we conclude that $u(\cdot,t) \leq \delta_0$ on $\mathbb{R}^n$, for $t \geq t_*$.
\end{proof}

We are now ready to show that the solution to the mean curvature flow converges to a constant. 
\begin{proposition}[Theorem \ref{thm:main} statement \eqref{eq:utoc0}]
The solution $u(\cdot, t)$ to the mean curvature flow \eqref{eq:mcf-nonparam} with initial condition $u(\cdot, 0)= u_0$ tends to a constant, denoted by $c_0$, uniformly in $\mathbb R^n$ as $t\to\infty$. 
\end{proposition}
It follows from the maximum principle that if such a constant $c_0$ exists, it cannot be bigger than $\delta_0<\ve$. 

\begin{proof}
The proof follows from results of Ecker--Huisken~\cite{EH} and the fact that $u(\cdot , t)$ is periodic in the spacial variables. Here we use the notation $Du$ for the Euclidean gradient and $\nabla u$ for the tangential gradient on the graph of $u$. In this notation we have the estimate
	\begin{equation}
	\label{eq:boundDu}
	\| \nabla (\sqrt{1+|D u |^2}) \| \leq \|A\| (1+|D u|^2),
	\end{equation}
where $A$ is the second fundamental form and $\| \cdot \|$ is the norm on the graph of $u$ ~\cite[proof of Corollary 4.2]{EH}. 

Now, using the periodicity of $u_0$, we know that $|Du_0|$ is bounded, and consequently $|Du(\cdot, t )|$ is bounded for all time by~\cite[Corollary 3.2]{EH}. Then applying~\cite[Proposition 4.4]{EH} we have that $\|A(\cdot, t)\| \to 0$ as $t\to\infty$. Because $|D u(\cdot, t )|$ is bounded for all time, the right-hand side of \eqref{eq:boundDu} goes to $0$, so that $|D u| \to c_1$ as $t\to \infty$ for some constant $c_1$. The periodicity implies that $c_1=0$ and gives the uniform convergence of $u$ to a constant as $t\to \infty$.
\end{proof}


\subsection{The average value of the initial entire graph}

The volume under the graph of $u_0$ restricted to the square $[-\frac{1}{2}, \frac{1}{2}]^n$ is 
 $\int_{[-\frac{1}{2}, \frac{1}{2}]^n} u_0 \, dvol =1$. A simple estimate shows that $$ B_{r-\sqrt{n}/2}(0) \subset \bigcup_{\substack{m \in  \mathbb{Z}^n, \\ m \in B_r(0)}} (m + [-\tfrac{1}{2}, \tfrac{1}{2}]^n) \subset B_{r+\sqrt{n}/2}(0).$$ Using these inclusions to estimate the integral of $u_0$ over unit cubes centered at points in the integer lattice, we have $$vol( B_{r-\sqrt{n}}(0) )  \leq \int_{B_r(0)} u_0 \, dvol  \leq vol( B_{r+\sqrt{n}}(0) ) ,$$ and it follows that $$\lim_{r \to \infty} \frac{\int_{B_r(0)} u_0 \, dvol}{vol( B_{r}(0) )} = 1.$$

The statement \eqref{eq:vto1} of Theorem \ref{thm:main} follows from the result of Repnikov--Eidelman \cite{repnikov-eidelman} for the heat equation. 

\section{An oscillating graph under mean curvature flow}

In this section, we expand upon the previous example to construct an example of an initial entire graph over $\mathbb{R}^n$, $n \geq 2$, that oscillates under the mean curvature flow and stabilizes to a constant under the heat flow. 

In the one-dimensional case, Nara--Taniguchi~\cite[Proposition 3.4]{NT} proved that there is a function $\phi_0(x) \in C^{2,\alpha}(\mathbb R)$ such that the solution $\phi(x,t)$ of~(\ref{eq:mcf-nonparam}) with $\phi(x,0) = \phi_0(x)$ does not stabilize. The heuristic idea  for the proof of Theorem~\ref{thm:mainosc} is to first take the one-dimensional example $\phi(x,t)$ and extend it to higher dimensions, simply by noticing that $\psi(x_1, \cdots,x_n,t) = \phi(x_1,t)$ is also a solution to the graphical mean curvature flow. Then by placing spikes with sufficient volume in the regions where the initial function $\phi(x,0)$ is zero, we can construct an initial function $w_0 \in C^{2,\alpha}(\mathbb{R}^n)$ that will stabilize to the constant 1 under the heat flow. However, since those spikes shrink very quickly under the mean curvature flow, the evolution of $w_0$ under the mean curvature flow mimics the  oscillatory behavior of $\phi(x,t)$.

We begin the proof of Theorem~\ref{thm:mainosc} by stating a scaled version of the example from the Nara--Taniguchi paper \cite{NT}. 

\begin{proposition} [Proposition 3.4 \cite{NT}]
\label{prop:NT}
Fix constants $\ell$, $a$, and $b$ with $0 < \ell \leq 1/2$ and $a<b$. Set $I_0 = [0, 1 -\ell]$ and $I_m = [m!+ \ell, (m+1)!-\ell]$, for $m \geq 1$. Let $\phi_0(x) = \phi_0^{\ell,a,b}(x) \in C^{2,\alpha}(\mathbb R)$ be a function that satisfies
\begin{enumerate}

\item[(1.)]  $a \leq \phi_0(x) \leq b$ for $x \in \mathbb R$.

\item[(2.)] $ \phi_0(x) = a$, for $|x| \in I_m$, $m \in \{1,3,5,\ldots\}.$

\item[(3.)] $ \phi_0(x) = b$, for $|x| \in I_m$, $m \in \{0,2,4,\ldots\}.$
\end{enumerate}
Then the solution $\phi(x,t)$ of~(\ref{eq:mcf-nonparam}) (the curve shortening flow) with $\phi(x,0) = \phi_0(x)$ satisfies 
	\[
	\liminf_{t\to \infty} \phi(0,t) = a, \quad \limsup_{t \to \infty} \phi(0,t) = b.
	\]
\end{proposition}

For later purpose,  let us consider the two families of slabs for $k \geq 1$
	\begin{align*}
	S_{2k} &= \{ x \in \mathbb R^n: |x_1| \in [ (2k)! -\tfrac{1}{4}, (2k+1)! + \tfrac{1}{4}]\}, \\
	S_{2k+1} &= \{ x \in \mathbb R^n: |x_1| \in [ (2k+1)! +\tfrac{1}{4}, (2k+2)! -\tfrac{1}{4}]\}.
	\end{align*}
Note that $S_{2k+1} \subset I_{2k+1}\times \mathbb R^{n-1}$ and $S_{2k} \supset I_{2k}\times \mathbb R^{n-1}$ if $\ell < \frac{1}{4}$.

\subsection{The initial graph $w_0(x)$}
Let $\ell_0$, $r_0$, and $\delta_0$ be the constants of the shrinking torus from Section \ref{sec:torus}. In addition, we assume without loss of generality that $r_0<1/4 - l_0$ and $\max (r_0, \delta_0)<\ve$.

We pick $a=0$, $b=1$, $\ell = \ell_0$ and a function $\phi_0$ that satisfies the conditions (1.)-(3.) of Proposition \ref{prop:NT}. We extend it to a function $\psi_0: \mathbb R^n \to \R$ by taking 
	\[
	\psi_0 (x) = \psi_0(x_1, x_2, \ldots, x_n) := \phi_0 (x_1).
	\]

Let $u_0(x)$ be the example from Theorem~\ref{thm:main}. Recall that $\int_{[-\frac{1}{2}, \frac{1}{2}]^n} u_0  = 1$. We define the function $w_0: \mathbb R^n \to [0,\infty)$ by
	\[
	w_0(x) = u_0(x) \left(\sum_{k=1}^{\infty} \chi_{S_{2k+1}}\right) + \psi_0(x),
	\]
where  $\chi_{A}$ is the characteristic function of the set $A$ (see picture). 

\def\JPicScale{0.73}
\ifx\JPicScale\undefined\def\JPicScale{1}\fi
\unitlength \JPicScale mm

\begin{picture}(162.5,68)(0,0)
\linethickness{0.7mm}
\qbezier(-0,30)(1.48,30)(1.56,30)
\qbezier(1.56,30)(1.64,30)(5,30)
\qbezier(5,30)(8.59,29.69)(9.38,31.25)
\qbezier(9.38,31.25)(10.16,32.81)(10,40)
\qbezier(10,40)(9.84,47.19)(10.62,48.75)
\qbezier(10.62,48.75)(11.41,50.31)(15,50)
\qbezier(15,50)(19.38,50)(27.5,50)
\qbezier(27.5,50)(35.62,50)(40,50)
\qbezier(40,50)(43.59,50.31)(44.38,48.75)
\qbezier(44.38,48.75)(45.16,47.19)(45,40)
\qbezier(45,40)(44.84,32.81)(45.62,31.25)
\qbezier(45.62,31.25)(46.41,29.69)(50,30)
\qbezier(50,30)(57.19,30)(85,30)
\qbezier(85,30)(112.81,30)(120,30)
\qbezier(120,30)(123.59,29.69)(124.38,31.25)
\qbezier(124.38,31.25)(125.16,32.81)(125,40)
\qbezier(125,40)(124.84,47.19)(125.62,48.75)
\qbezier(125.62,48.75)(126.41,50.31)(130,50)
\qbezier(130,50)(134.61,50)(143.44,50)
\qbezier(143.44,50)(159.27,50)(162,50)
\linethickness{0.3mm}
\multiput(10,12.5)(0,2){23}{\line(0,1){1}}
\linethickness{0.3mm}
\multiput(45,12.5)(0,2){23}{\line(0,1){1}}
\linethickness{0.3mm}
\multiput(125,12.5)(0,2){23}{\line(0,1){1}}
\linethickness{0.3mm}
\multiput(6.25,16.25)(0,3.93){11}{\line(0,1){1.96}}
\linethickness{0.3mm}
\multiput(13.75,25)(0,3.82){9}{\line(0,1){1.91}}
\linethickness{0.3mm}
\multiput(41.25,25)(0,3.82){9}{\line(0,1){1.91}}
\linethickness{0.3mm}
\multiput(48.75,16.25)(0,3.93){11}{\line(0,1){1.96}}
\linethickness{0.3mm}
\multiput(121.25,16.25)(0,3.93){11}{\line(0,1){1.96}}
\linethickness{0.3mm}
\multiput(128.75,25)(0,3.82){9}{\line(0,1){1.91}}
\linethickness{0.3mm}
\put(48.75,25){\line(1,0){72.5}}
\put(121.25,25){\vector(1,0){0.12}}
\put(48.75,25){\vector(-1,0){0.12}}
\linethickness{0.3mm}
\put(6.25,20){\line(1,0){42.5}}
\put(48.75,20){\vector(1,0){0.12}}
\put(6.25,20){\vector(-1,0){0.12}}
\linethickness{0.3mm}
\put(0,30){\line(1,0){162.5}}
\put(162.5,30){\vector(1,0){0.12}}
\linethickness{0.7mm}
\linethickness{0.3mm}
\put(121.25,55){\line(1,0){3.75}}
\put(125,55){\vector(1,0){0.12}}
\put(121.25,55){\vector(-1,0){0.12}}
\put(86.25,21.25){\makebox(0,0)[cc]{$S_{2k+1}$}}

\put(72.5,21.25){\makebox(0,0)[cc]{}}

\put(23.75,17.5){\makebox(0,0)[cc]{$S_{2k}$}}

\put(122.5,60){\makebox(0,0)[cc]{$\frac{1}{4}$}}

\put(85,41.25){\makebox(0,0)[cc]{Add the graph of $u_0$,} }
\put(85, 36.50){\makebox(0,0)[cc]{i.e. add spikes here.}}

\put(66.25,48.75){\makebox(0,0)[cc]{}}

\put(160,26.25){\makebox(0,0)[cc]{$x_1$}}

\put(10,10){\makebox(0,0)[cc]{$(2k)!$}}

\put(45,10){\makebox(0,0)[cc]{$(2k+1)!$}}

\put(125,10){\makebox(0,0)[cc]{$(2k+2)!$}}

\linethickness{0.7mm}

\put(155,55){\makebox(0,0)[cl]{$\psi_0$}}

\linethickness{0.3mm}
\put(121.25,20){\line(1,0){41.25}}
\put(121.25,20){\vector(-1,0){0.12}}
\put(142.5,17.5){\makebox(0,0)[cc]{$S_{2k+2}$}}

\put(85,21.25){\makebox(0,0)[cc]{}}

\linethickness{0.3mm}
\put(125,55){\line(1,0){3.75}}
\put(128.75,55){\vector(1,0){0.12}}
\put(125,55){\vector(-1,0){0.12}}
\linethickness{0.3mm}
\put(45,55){\line(1,0){3.75}}
\put(48.75,55){\vector(1,0){0.12}}
\put(45,55){\vector(-1,0){0.12}}
\linethickness{0.3mm}
\put(41.25,55){\line(1,0){3.75}}
\put(45,55){\vector(1,0){0.12}}
\put(41.25,55){\vector(-1,0){0.12}}
\linethickness{0.3mm}
\put(10,55){\line(1,0){3.75}}
\put(13.75,55){\vector(1,0){0.12}}
\put(10,55){\vector(-1,0){0.12}}
\linethickness{0.3mm}
\put(6.25,55){\line(1,0){3.75}}
\put(10,55){\vector(1,0){0.12}}
\put(6.25,55){\vector(-1,0){0.12}}
\put(127.5,60){\makebox(0,0)[cc]{$\frac{1}{4}$}}

\put(46.25,60){\makebox(0,0)[cc]{$\frac{1}{4}$}}

\put(42.5,60){\makebox(0,0)[cc]{$\frac{1}{4}$}}

\put(11.25,60){\makebox(0,0)[cc]{$\frac{1}{4}$}}

\put(7.5,60){\makebox(0,0)[cc]{$\frac{1}{4}$}}

\end{picture}

\subsection{Evolution by heat flow}
We note that 
	\begin{equation}
	\label{eq:average}
	 \lim_{r \to \infty} \frac{\int_{B_r(x)} w_0 \, dvol}{vol( B_{r}(x) )} = 1,
	 \end{equation}
with uniform convergence for $x \in \mathbb R^n$ by construction. Note that the existence of the limit \eqref{eq:average} at $x=0$ is a necessary and sufficient condition for the convergence of $v(\cdot, t)$ on compact subsets of $\mathbb R^n$ as $t \to \infty$ \cite{repnikov-eidelman}. The uniform convergence of the limit \eqref{eq:average} is equivalent to the uniform convergence of $v(\cdot, t)$ to 1 as $t\to \infty$ \cite{kamin}. 

\subsection{Barriers for the mean curvature flow}

Let $w(x,t)$ be the solution of~(\ref{eq:mcf-nonparam}) with $w(x,0) = w_0(x)$. Also, let $\mathbb S_0$ be the shrinking sphere with extinction time $t_*$. The radius of such a sphere is denoted by $\rho_0$ and we know from \cite{Ang} that $\rho_0<r_0$. 

We use spheres in $\mathbb{R}^{n+1}$ centered at points $(x, \rho_0)$ with $x \in \partial S_{2k+1}$ as barriers to obtain
\[
w(x,t_*) < \rho_0, \quad x \in \partial S_{2k+1}.
\]
Because the solution $w(x,t)$ is periodic in $x_2, x_3, \ldots, x_n$, an argument similar to the one from Propositions \ref{prop:2}  and \ref{prop:3} for the punctured slabs $\Omega_t \cap S_{2k+1}$   gives us that 
\[
w(x, t_*) < \max (\rho_0, \delta_0)<\ve, \quad  x \in S_{2k+1}.
\]
By placing spheres above $w_0$ in $S_{2k}$, i.e. spheres centered at $(x, 1+\rho_0)$, for $x \in S_{2k}$, we obtain
\[
w(x,t_*) <1 + \rho_0, \quad x \in S_{2k}.
\]

It follows that at time $t_*$, we can construct an upper barrier for $w(x,t_*)$ by taking $w^+(x,t_*)=\phi^+_0(x_1)$, where $\phi_0^+ \in  C^{2,\alpha}(\mathbb R)$ and satisfies
	\begin{gather*}
	\ve \leq \phi^+_0(x_1) \leq 1+\rho_0\\
	\phi^+_0 (x_1) = 1+\rho_0, \quad |x_1| \in [ (2k)! -\tfrac{1}{4}, (2k+1)! + \tfrac{1}{4}]\\
	\phi^+_0(x_1) = \ve, \quad |x_1| \in [ (2k+1)! +\tfrac{1}{2}, (2k+2)! -\tfrac{1}{2}]
	\end{gather*}
By Proposition \ref{prop:NT}, the solution $\phi^+(x_1,t)$ to the curve shortening flow with initial data $\phi^+(x_1,0) = \phi^+_0(x_1)$ oscillates and $\liminf_{t\to\infty} \phi^+(0,t) \leq \ve$. The function $w^{+}(x,t) = \phi^+(x_1,t-t_*)$ is a solution to the graphical mean curvature flow and an upper barrier for $w(x,t)$ for $t >t_*$. We therefore have
	\[
	\liminf_{t\to \infty} w(0,t) \leq \liminf_{t\to \infty} w^+(0,t) \leq \ve.
	\]
The function $\psi(x,t)$, which is the solution to \eqref{eq:mcf-nonparam} with initial data $\psi_0(x)$, serves as a lower barrier and we obtain
	\[
	\limsup_{t\to\infty} w(0,t) \geq \limsup_{t\to \infty} \psi(0,t) = 1. 
	\]
This completes the proof of Theorem~\ref{thm:mainosc}.

\def\cprime{$'$}

\end{document}